\documentclass[reqno]{amsart}

\usepackage{amsfonts}
\usepackage{amssymb}
\usepackage{color}
\usepackage{amsthm}
\usepackage{amsmath}
\usepackage{enumerate}
\usepackage[usenames,dvipsnames]{xcolor}
\usepackage{geometry}
\usepackage{enumitem}

\newtheorem{thm}{Theorem}
\newtheorem{lm}{Lemma}
\newtheorem{prop}{Proposition}
\newtheorem{ex}{Example}

\theoremstyle{definition}

\newtheorem{df}{Definition}

\theoremstyle{remark}
\newtheorem{rem}{Remark}

\newcommand{\R}{\mathbb{R}}

\newcommand{\T}{\mathbb T}
\newcommand{\N}{\mathbb N}
\newcommand{\Z}{\mathbb Z}

\newcommand{\A}{\mathbb A}

\title{On the degenerate Arnold conjecture on $\mathbb T^{2m}\times \mathbb C\mathbb P^n$}

\author[L. Asselle]{Luca Asselle}
\address{Ruhr-Universit\"at Bochum, Fakult\"at f\"ur Mathematik, Universit\"atsstr. 150, 44801, Bochum, Germany}
\email{luca.asselle@rub.de}

\author[M. Starostka]{Maciej Starostka}
\address{Gda\'nks University of Technology, Gabriela Narutowicza 11/12, 80233 Gda\'nsk, Poland \newline
Institut f\"ur Mathematik, Naturwissenschaftliche Fakult\"at II, Martin-Luther-Universit\"at Halle-Wittenberg, 06099 Halle (Saale), Germany}
\email{maciejstarostka@pg.edu.pl}
 
\begin{document}

\maketitle

\begin{abstract}
In the 1960s Arnold conjectured that a Hamiltonian diffeomorphism of a closed connected symplectic manifold $(M,\omega)$ should have at least 
as many contractible fixed points as a smooth function on $M$ has critical points. Such a conjecture can be seen 
as a natural generalization of Poincar\'e's last geometric theorem and is one of the most famous (and still nowadays open in its full generality) problems in symplectic geometry. 
In this paper, we build on a recent approach of the authors and Izydorek to the Arnold conjecture on $\mathbb C\mathbb P^n$  to
show that the (degenerate) Arnold conjecture holds for Hamiltonian diffeomorphisms $\phi$  of $\T^{2m}\times \mathbb C\mathbb P^n$, $m,n\geq 1$, which are $C^0$-close to the identity in the $\mathbb C \mathbb P^n$-direction, namely that any such $\phi$ has at least $\text{CL}(\T^{2m}\times \mathbb C\mathbb P^n)+1= 2m+n+1$ contractible fixed points. 
\end{abstract}


\vspace{3mm}

\begin{small} 
\noindent \textbf{Data availability statement:} Data sharing not applicable to this article as no datasets were generated or analysed during the current study.
\end{small}


\section{Introduction}

Given a closed symplectic manifold $(M,\omega)$ and a smooth time dependent one-periodic Hamiltonian function $H:\T\times M\to \R, \T := \R/\Z,$ one defines the Hamiltonian flow $\phi_t$ by\footnote{Actually, $\phi_t$ is a flow only if $H$ is autonomous.}
$$\left \{ \begin{array}{l} \displaystyle \frac{\mathrm d}{\mathrm d t} \phi_t = X_H (\phi_t),\\ \ \ \ \, \phi_0 = \text{id},\end{array}\right .$$ 
where $X_H$ is the (time-dependent) Hamiltonian vector field defined by 
$$\iota_{X_H} \omega = \omega(X_H,\cdot) = - \mathrm d H.$$

A diffeomorphism $\phi$ of $M$ is called \textit{Hamiltonian} if there exists a time-dependent $H:\T\times M\to \R$ such that $\phi=\phi_1$, i.e. the time-one-Hamiltonian flow of $H$ is precisely $\phi$.
A fixed point $x\in M$ of a Hamiltonian diffeomorphism $\phi$ is called \textit{contractible} if the loop 
 $$[0,1]\ni t\mapsto \phi_t(x)$$ 
 is contractible in $M$. The (degenerate) \textit{Arnold conjecture}, formulated by Arnold in 1960s, asserts that the number of contractible\footnote{In general, one cannot expect the existence of non-contractible fix points. For instance, this is the case for Hamiltonians which are sufficiently $C^1$-small.} fix points of any Hamiltonian 
 diffeomorphism $\phi$ of $(M,\omega)$ is bounded from below 
 by the smallest positive integer $k\in \N$ such that there exists a smooth function $f:M\to \R$ with exactly $k$ critical points: 
 \begin{equation}
| \text{Fix} (\phi) | \geq \text{crit}(M):= \min\big \{ |\text{crif}(f)| \, |\, f:M\to \R \ \text{smooth}\big \}.
\label{eq:ac1}
\end{equation}

Under the additional assumption that all contractible fixed points are non-degenerate, the lower bound should be given by the smallest number of critical points of a Morse function on $M$: 
 \begin{equation}
| \text{Fix} (\phi) | \geq \text{Morse}(M):= \min\big \{ |\text{crif}(f)| \, |\, f:M\to \R \ \text{Morse function}\big \}.
\label{eq:ac1nondeg}
\end{equation}
In both formulations above the conjecture is still widely open. The first result appeared with the work of Conley and Zehnder \cite{Conley}, who proved \eqref{eq:ac1} and \eqref{eq:ac1nondeg} for the standard $2m$-dimensional torus. 
As far as \eqref{eq:ac1} is concerned, the most general answer appears in the work of Rudyak-Oprea \cite{Rudyak:99}, which confirmed the validity of \eqref{eq:ac1}
for any closed symplectically aspherical manifold $(M,\omega)$. The weaker version of \eqref{eq:ac1nondeg} known as the \textit{homological Arnold conjecture} asserting that, if all contractible fix points of $\phi$ are non-degenerate, then $|\text{Fix}(\phi)|$ is bounded below by the sum of the Betti numbers of $M$, 
has received great attention starting from the pioneering work of Floer \cite{Floer:88}, and it is now known to hold in general, see \cite{Floer:89,Fukaya:1,Fukaya:2,Liu}.

Over the past decades, many (non-equivalent) versions of \eqref{eq:ac1} have been stated. For the purposes of this introduction we only recall the following. By Lyusternik-Schnirelmann theory, $\text{crit}(M)$ is bounded from below by the Lyusternik-Schnirelmann category $LS(M)$ of $M$, that is, the smallest number of contractible open subsets of $M$ that one needs to cover $M$ completely. In general, $LS(M)$ is hard to compute, but one can bound it from below by $\text{CL}(M)+1$, that is, one plus the cup-length of $M$. In many examples, one has the equality $LS(M)=CL(M)+1$, but in general the two quantities can be arbitrarily far from each other, see \cite{CLOT}. Therefore, the degenerate Arnold conjecture \eqref{eq:ac1} is often formulated in the weaker version
 \begin{equation}
| \text{Fix}(\phi) | \geq \text{CL}(M)+1.
\label{eq:ac2}
\end{equation}
For non symplectically aspherical manifolds, Equation \eqref{eq:ac2} has been proved only in particular cases, e.g. for $\mathbb C\mathbb P^n$ endowed with the Fubini-Study form \cite{Fo}, and is in general still open even for simple manifolds as the product $\T^{2m}\times \mathbb C\mathbb P^n$, $m,n\geq 1$, with its standard symplectic form. In this case, the best known lower bound is due to Oh \cite{Oh} and states that 
$| \text{Fix}_{\text{contr}} (\phi) |$ is bounded from below by $n+1=\text{CL}(\mathbb C\mathbb P^n)+1$. In \cite{Oh}, the author also raises doubts on the validity of the degenerate Arnold conjecture (in any of its form) in full generality. Going in the opposite direction, in the present paper we build on a recent, purely based on Conley theory (suitably revisited), approach to the Arnold conjecture on $\mathbb C\mathbb P^n$ introduced by the authors and Izydorek in \cite{AIS} to prove \eqref{eq:ac2} for Hamiltonian diffeomorphisms on $\mathbb T^{2m}\times \mathbb C\mathbb P^n$ generated by $C^0$-small Hamiltonians $H:\T \times \mathbb T^{2m}\times \mathbb C\mathbb P^n\to \R$ which are $C^1$-small in the $\mathbb C\mathbb P^n$-direction. 

\begin{thm}
\label{thm:a}
On the closed symplectic manifold $( \mathbb T^{2m}\times \mathbb C\mathbb P^n, \omega_{\mathrm{std}}\oplus \omega_{FS})$ consider a Hamiltonian function $H_0:\T \times \mathbb T^{2m}\times \mathbb C\mathbb P^n\to \R$
such that $\max \{2\|H_0\|_\infty,\|\nabla_{\mathbb C\mathbb P^n} H_0\|_\infty\} < 1/2$, and denote by $\phi$ the corresponding Hamiltonian diffeomorphism. Then 
$$| \mathrm{Fix} (\phi) | \geq 2m+n+1 = \mathrm{CL}(\T^{2m}\times \mathbb C\mathbb P^n)+1.$$
\end{thm}

To finish this introduction we give a quick account on the proof of Theorem \ref{thm:a}. The starting point is that the desired fix points can be characterized variationally as the critical points of a suitable modification of the Hamiltonian action functional
$$\A_H : H^{1/2}(\T,\R^{2m}\times \R^{2n+2})\times \R\to \R$$
associated with the canonical lift $H$ of $H_0$ to $\R^{2m}\times \R^{2n+2}$ (see Section 2.1 for more details). The additional $\R$-factor takes into account the fact that 1-periodic orbits of Hamilton's equations in $\T^{2m}\times \mathbb C\mathbb P^n$ lift to solutions of Hamilton's equation on $\R^{2m}\times \R^{2n+2}$ which a priori only start and end at the same fibre of the Hopf-fibration $S^{2n+1}\to \mathbb C\mathbb P^n$. Under the assumptions of Theorem \ref{thm:a}, we can then show that the gradient flows of $\A_H$ and $\A_0$, the modified Hamiltonian action associated with the zero-Hamiltonian, are related by continuation within a suitable 
bounded set $\Omega\subset H^{1/2}(\T,\R^{2m}\times \R^{2n+2})\times \R$. By the continuation invariance of the Conley index we are thus left to compute the (relative) cup-length of the Conley index, which gives a lower bound on the number of critical points, for the zero Hamiltonian. In such a case we have a product gradient flow, and hance we can treat the two factors separately. This yields the desired estimate. 

\vspace{2mm}

\noindent \textbf{Acknowledgments:} L.A. is partially supported by the DFG grant 540462524 ``Morse theoretical methods in Analysis, Dynamics, and Geometry''. M.S. is partially supported by the NCN grant 2023/05/Y/ST1/00186 ``Morse theoretical methods in Analysis, Dynamics, and Geometry''. 


\section{Preliminaries}

\subsection{Hamiltonian systems on $\T^{2m}\times \mathbb C\mathbb P^n$}
Let $H_0:\T\times \T^{2m}\times \mathbb C\mathbb P^n\to \R$ be a smooth time-depending one-periodic Hamiltonian on the symplectic manifold $(\T^{2m}\times \mathbb C\mathbb P^n, \omega_0 \oplus \omega_{FS})$,
where $\omega_0$ denotes the standard symplectic form on $\T^{2m}$ and $\omega_{FS}$ the Fubini-Study symplectic form on $\mathbb C\mathbb P^n$. The Arnold conjecture for $\T^{2m}\times \mathbb C\mathbb P^n$ states that the number of contractible one-periodic solutions, hereafter called \textit{solutions} for simplicity, to 
\begin{equation}
\dot z(t) = J \nabla H_0(t,z(t))
\label{eq:hameq1}
\end{equation}
is bounded from below by $2m+n+1=\text{CL}(\T^{2m}\times \mathbb C\mathbb P^n)+1$. Here $z=(x,y)\in \T^{2m}\times \mathbb C\mathbb P^n$. 

Solutions to \eqref{eq:hameq1} can be studied by lifting the Hamiltonian system to $\R^{2m}\times \R^{2n+2}$ as follows: we first view $H_0$ as a $\Z^{2m}$-periodic Hamiltonian function on $\T\times \R^{2m}\times \mathbb C\mathbb P^n$, and then lift $H_0$ to a Hamiltonian $\T\times \R^{2m}\times S^{2n+1}$ which is invariant under the Hopf action on $S^{2n+1}$, denoted $H_1$, and finally extend $H_1$ to $\T\times \R^{2m}\times \R^{2n+2}$ quadratically:
$$H:\T\times \R^{2m}\times \R^{2n+2}\to \R, \quad H(t,x,y):= |y|^2 \cdot H_1 (t,x,|y|^{-1}y).$$ 
With slight abuse of notation we write also $z=(x,y) \in \R^{2m}\times \R^{2n+2}$ and denote again by $J$ the complex structure on $\R^{2m}\times \R^{2n+2}$ which is obtained from the standard complex structures on $\R^{2m}$ and $\R^{2n+2}$. As $H$ is $\Z^{2m}$-periodic in the $\R^{2m}$-direction, quadratic in the $\R^{2n+2}$-direction, and invariant under the Hopf action, the Hamiltonian vector field $J\nabla H$ (more precisely, its restriction to $T(\R^{2m}\times S^{2n+1})$) projects to the Hamiltonian vector field $J\nabla H_0$. In particular, solutions to 
\begin{equation}
\dot z(t) = J \nabla H(t,z(t))
\label{eq:hameq2}
\end{equation}
for which $|y(t)|\equiv 1$ project to solutions of \eqref{eq:hameq1}. The converse is in general not true, and indeed solutions of \eqref{eq:hameq1} lift to paths which solve \eqref{eq:hameq2} and satisfy $|y(t)|\equiv 1$, but a priori only start and end at the same Hopf-fiber. To overcome this issue we consider, for arbitrary $\lambda \in \R$, the solutions to
\begin{equation}
\left \{\begin{array}{l} \dot x(t) = J \nabla_x H(t,x(t),y(t)), \\ \dot y(t) = J \big ( \nabla_y H(t,x(t),y(t)) - 2\pi \lambda y(t)\big).\end{array}\right.
\label{eq:hameq3}
\end{equation}
Indeed, $z(t)=(x(t),y(t))$ is a solution of \eqref{eq:hameq3} if and only if 
$$\tilde z(t) := (x(t), e^{2\pi \lambda Jt}y(t))$$ 
satisfies \eqref{eq:hameq2} and $\tilde z(1)=(x(1),y(1))=(x(0),e^{2\pi \lambda J}y(0)).$ We observe that, if $z$ is a solution of \eqref{eq:hameq3}, then $|y (t)|$ is constant. Since $\|y\|_2 = |y(t)|^2$, we are only interested in solutions of \eqref{eq:hameq3} for which $\|y\|_2=1$, as these are the only ones projecting to solutions on $\T^{2m}\times \mathbb C\mathbb P^n$.

Notice finally that the correspondence between solutions of \eqref{eq:hameq1} and of \eqref{eq:hameq3} is not one-to-one. Indeed, for $\tilde z=(\tilde x,\tilde y)$ solution of \eqref{eq:hameq1} choose a lift $z=(x,y)$ to $\R^{2m}\times \R^{2n+2}$. Clearly, $z$ satisfies \eqref{eq:hameq3} for some $\lambda \in \R$. Now, for any $\mathbf{k}\in \Z^{2m}$, $\theta\in \R$, and $k\in\mathbb Z$, the curve 
$$z_{\mathbf{k},\theta,k}:t\mapsto \big (\mathbf{k} + x(t), e^{2\pi J(\theta+kt)}\cdot y(t)\big )$$
satisfies \eqref{eq:hameq3} with $\lambda_{\mathbf{k},\theta,k}=\lambda +k$ and projects to $\tilde z$. The $\mathbb Z^{2m}$-translation $\mathbf{k}$ is easily ruled out by assuming that $x$ is of the form $x_0 + \bar x$, where $x_0\in \mathbb T^{2m}$ and $\bar x$ is a loop in $\R^{2m}$ with zero mean. Therefore, the Arnold conjecture in this setting amounts to saying that the problem 
\begin{equation}
\left \{\begin{array}{l}  \dot x(t) = J \nabla_x H(t,x(t),y(t)), \\ \dot y(t) = J \big (\nabla_y H(t,x(t),y(t)) - 2\pi \lambda y(t)\big ), \\ \lambda \in [\lambda_0,\lambda_0+1), \\ x = x_0 + \bar x, \ \text{with}\ x_0\in \T^{2m}, \ \int_{\R^{2m}} \bar x (t)\, \mathrm{d} t =0,\\ \|y\|_2=1,\end{array} \right . 
\label{eq:hameq4}
\end{equation}
for some $\lambda_0\in \R$, has at least $2m+n+1$ distinct $S^1$-families of solutions. Solutions of \eqref{eq:hameq4} admit a variational characterization as critical points of a suitable modification of the Hamiltonian action functional as we now recall. 
Clearly, we can assume that $\lambda_0$ is such that \eqref{eq:hameq4} has no solutions with $\lambda=\lambda_0$, as otherwise we would immediately have the existence of infinitely many distinct solutions, hence of infinitely many solutions of \eqref{eq:hameq1}. 
We set 
$$F_\lambda (x,y) := \big (-J\dot x -\nabla_x H(x,y), - J\dot y - \nabla_y H(x,y) + 2\pi \lambda y\big ), \quad \forall \lambda \in \R,$$
and denote by 
$$\jmath^*:L^2(\T,\R^{2m}\times \R^{2n+2}) \to H^{1/2}(\T,\R^{2m}\times \R^{2n+2})=:\mathbb H$$ 
the adjoint operator to the canonical inclusion. In what follows, $\|\cdot\|$ will denote  the $H^{1/2}$-norm. Recall that 
$$\jmath^* \Big (-J \frac{\mathrm d}{\mathrm d t}\Big ) (x,y) = (x^+-x^-,y^+-y^-) = : L(x,y)$$
where $(x,y) = (x^0+x^++x^-,y^0+y^++y^-)\in \R^{2m}\times \R^{2n+2} \times \mathbb H^+ \times \mathbb H^-$ is the canonical decomposition of $(x,y)$ into positive, negative, and zero Fourier modes (see \cite{HZ} for more details). 

We introduce now the following additional condition:
\begin{enumerate}
\item[(A)] There exists $\epsilon>0$ such that 
$$\|\jmath^* \text{pr}_2 F_\lambda (x,y)\| \geq c \|y\|, \quad \forall (x,y) \in \mathbb H,$$  
for all $\lambda \in (\lambda_0-\epsilon, \lambda_0+\epsilon)$. 
\end{enumerate} 
Here, $\text{pr}_2:\mathbb H\to \mathbb H$ denotes the projection on the second factor. We will comment more on such a condition later. For the moment we only stress that Condition (A) implies (but it is far from being equivalent to)  that the second equation in \eqref{eq:hameq4} has no non-zero solutions for $\lambda =\lambda_0$, and hence in particular that \eqref{eq:hameq4} has no solutions for $\lambda=\lambda_0$.

Assume that Condition (A) hold and consider a smooth function $\chi:\R\to \R$ such that:
\begin{enumerate}
\item[(i)] $\chi \equiv \lambda_0$ on $(-\infty,\lambda]$, $\chi \equiv \lambda_0+1$ on $[\lambda_0+1,+\infty)$,
\item[(ii)] $\chi(\lambda)= \lambda$ on $(\lambda_0+\epsilon/2, \lambda_0+1-\epsilon/2)$,
\item[(iii)] $\chi'>0$ on $(\lambda_0,\lambda_0+1)$,
\end{enumerate}
and define the modified Hamiltonian action functional $\mathbb A_H : \mathbb H\times \R\to \R$ by 
$$\mathbb A_H (x,y,\lambda) := \frac 12 \langle - J\dot x,x\rangle_2 + \frac12 \langle -J\dot y, y \rangle_2 - \int_0^1 H(t,x(t),y(t))\, \mathrm d t + \pi \big (\chi (\lambda) \| y\|_2^2 -\lambda\big ).$$ 
A straightforward computation shows that the gradient of $\mathbb A_H$ with respect to the product metric on $\mathbb H\times \R$ is given by
$$\nabla \mathbb A_H (x,y,\lambda) = \Big ( \jmath^* F_{\chi(\lambda)} (x,y), \pi \big ( \chi'(\lambda) \|y\|_2^2 -1 \big ) \Big ).$$ 
Notice that $\mathbb A_H$ is invariant under the $\Z^{2m}$-action. Therefore, it descends to a functional, again denoted with $\A_H$, 
$$\A_H : \T^{2m}\times R^{2n+2} \times \mathbb H^+\times \mathbb H^-\to \R.$$

\begin{lm}
Assume that Condition $\mathrm{(A)}$ holds. Then, the pair $(x,y)$ is a solution of \eqref{eq:hameq4}, for some $\lambda \in (\lambda_0,\lambda_0+1)$, if and only if $(x,y,\lambda)$ is a critical point of $\mathbb A_H$. 
\end{lm}
\begin{proof}
If $(x,y)$ is a solution of \eqref{eq:hameq4} for some $\lambda \in (\lambda_0,\lambda_0+1)$, then by Assumption (A) we have that $\lambda \in (\lambda_0+\epsilon, \lambda_0-\epsilon)$. Property (ii) now implies that 
$$\nabla \mathbb A_H (x,y,\lambda) = \Big ( \jmath^* F_{\chi(\lambda)} (x,y), \pi \big ( \chi'(\lambda) \|y\|_2^2 -1 \big ) \Big ) = \Big ( \jmath^* F_{\lambda} (x,y), \pi \big ( \|y\|_2^2 -1 \big ) \Big ) = (0,0,0),$$
that is, $(x,y,\lambda)$ is a critical point of $\mathbb A_H$. Conversely, let $(x,y,\lambda)$ be a critical point of $\mathbb A_H$. Since
$$\chi'(\lambda) \|y\|_2^2 -1 =0,$$
by Property (i) we deduce that $\lambda \in (\lambda_0,\lambda_0+1)$ and $y\neq 0$. If now $\lambda \in (\lambda_0,\lambda_0+\epsilon)$, then Properties (ii) and (iii) imply that $\chi(\lambda) = \tilde \lambda$ for some $\tilde \lambda \in (\lambda_0,\lambda_0+\epsilon)$. But then $\text{pr}_2 F_{\tilde \lambda} (x,y) =0$ is incompatible with Condition (A). Similarly, we exclude that $\lambda \in (\lambda+1-\epsilon,\lambda_0+1)$. Therefore, $\lambda\in (\lambda_0+\epsilon,\lambda_0+1-\epsilon)$, and the claim follows by Property (ii). 
\end{proof}

In virtue of the lemma above, the degenerate Arnold conjecture on $\T^{2m}\times \mathbb C\mathbb P^n$ follows if we can show the existence of at least $2m+n+1$ distinct $S^1$-families of critical points for the modified Hamiltonian action $\mathbb A_H$, provided we can find $\lambda_0\in \R$ such that Condition (A) is satisfied. Unfortunately, for an arbitrary Hamiltonian $H$ there is no reason why Condition (A) should be satisfied, and indeed it is even easy to construct examples where the second equation in \eqref{eq:hameq4} has solutions different from $y\equiv 0$ for any $\lambda \in \R$. Consider for instance the Hamiltonian 
\begin{equation}
H:\T\times \R^{2m}\times \R^{2n+2}\to \R, \quad H(t,x,y):=\frac \pi 2 \sin(x)|y|^2.
\label{eq:counterexample}
\end{equation}
A straightforward computation shows that, in this case, the second equation in \eqref{eq:hameq4} reads 
$$\dot y = \pi J\big (\sin (x) - 2\lambda \big ) y.$$
Hence, for every $\lambda \in [-1/2,1/2]$, every constant loop $y\equiv y_0$ is a solution provided $x\equiv x_0$ is chosen so that $\sin (x_0)=2\lambda$. More generally, for any $\lambda \in [-1/2-k,-1/2-k+1], k \in \Z$, every loop $y(t) := e^{2\pi kJt}y_0$ is a solution with $x\equiv x_0$ so that $\sin (x_0)=2(k+\lambda)$.

The starting point of the present paper is the observation that, if $H$ is $C^0$-small and $C^1$-small in the $\mathbb C\mathbb P^n$-direction in the sense of Theorem \ref{thm:a}, then Condition (A) is satisfied with 
$\lambda_0=- 1/2$. The functional $\mathbb A_H$ belongs to the class of functionals defined on a Hilbert space whose gradient is a \textit{LS-vector field}, namely 
of the form ``compact perturbation of a fixed quadratic form''. Over the past decades, such a class of functionals has been subject of a thorough study, and many constructions which were previously available only
on locally compact metric spaces have been extended to such a setting. Famous examples are the Brouwer degree and the Conley index. The latter will be the main ingredient in the proof of Theorem \ref{thm:a}. For this reason, in the next subsections we will briefly introduce the Conley index and its relative cup-length first in locally compact metric spaces and then explain how these notions can be extended to LS-flows.  


\subsection{The Conley index.} Let $\varphi:\R\times X\to X$ be a flow  on a locally compact metric space $X$. The (homotopy) Conley index, as introduced by Conley in his seminal work \cite{Conley:78}, is an important tool to analyze the dynamics of $\varphi$ (and, in general, of discrete or continuous dynamical systems) as it allows to prove the existence of special orbits, like stationary, periodic, or heteroclinic orbits, or to prove chaotic behavior of the system by showing that the homotopy type of suitably defined pointed topological spaces is non trivial. 

More precisely, we call a closed bounded subset $S\subset X$ an \textit{isolated invariant set} for $\varphi$ if there exists a closed bounded neighborhood $N$ of $S$ in $X$ such that $S$ is the $\varphi$-invariant part of $N$: 
$$S = \text{inv} (N) := \big \{x\in N \ \big |\ \varphi_t(x) \in N, \ \forall t \in \R\big \}.$$ 
In this case, $N$ is called an \textit{isolating neighborhood} of $S$. An \textit{index pair} $(N,L)$ for an isolated invariant set $S$ consists of closed bounded subsets $L\subset N$ of $X$ such that:
\begin{itemize}
\item $\overline{N\setminus L}$ is an isolating neighborhood of $S$.
\item $L$ is positively invariant in $N$: $\varphi_t(x) \in L$ for all $x\in L$ and $t>0$ such that $\varphi([0,t],x)\subset N$.
\item $L$ is an exit set for $N$: if $x\in N$ and $t>0$ are such that $\varphi_t(x)\notin N$, then there exists $t_0\in [0,t]$ such that $\varphi_{t_0}(x) \in L$. 
\end{itemize}
Given an isolated invariant set $S$, it can be shown that index pairs exist. Moreover, if $(N,L)$ and $(N',L')$ are index pairs for $S$, then the quotient spaces $N/L$ and $N'/L'$ are homotopy equivalent with base points $[L]$ and $[L']$ fixed. 

\begin{df}
Let $S$ be an isolated invariant set for the flow $\varphi:\R\times X\to X$ on the locally compact metric space $X$. The \textit{Conley index} $h(S):= h(S,\varphi)$ is the homotopy type of the pointed space $(N/L,[L])$, where $(N,L)$ is an index pair for $S$. 
\end{df}

\begin{ex}
On $X=\R^n$ consider the linear flow $\varphi_t(x) = e^{tA}x$, where $A= \mathrm{diag} (\lambda_1,...,\lambda_n)$ is a diagonal matrix with 
$$\lambda_1 >...>\lambda_k>0>\lambda_{k+1}>...>\lambda_n.$$
The origin is a hyperbolic fixed point of $\varphi$ and $S=\{0\}$ is an isolated invariant set. $B^{k}\times B^{n-k}$ (and, more generally, any compact neighborhood of $S$)  is an isolating neighborhood of $S$, and
$(B^k\times B^{n-k},S^{k-1}\times B^{n-k})$ is an index pair for $S$. Since $B^{n-k}$ is contractible, the Conley index of $S$ is the homotopy type of $(B^k/S^{k-1},[S^{k-1}])$, which is the same as the homotopy type of 
$(S^k,\ast)$. In this example, one recovers the Morse index of the hyperbolic fixed points. 
\end{ex}

In applications, one usually first has a set $N$ which is an isolating neighborhood of some a priori unknown isolated invariant set $S:= \text{inv}(N)$. Then, one tries to compute $h(S)$ or to obtain some information, like its homology groups. For this, a crucial role is played by the invariance of the Conley index under deformations (called \textit{continuation}) of the flow which do not change the isolating neighborhood $N$, as 
the computation of the index can then be reduced to cases which are well understood. Finally, one can use the knowledge about $h(S)$ in order to investigate the invariant set $S$ itself. Whereas one can immediately deduce that $S$ is non-empty if $h(S)$ is not trivial, additional information on the flow inside $N$ is needed to obtain more detailed results about $S$, for instance that $S$ contains a periodic orbit. 

The definition given in \cite{Conley:78} works only for flows on locally compact metric spaces. For our purposes, we need a generalization of the Conley index to LS-flows on a Hilbert space $\mathbb H$ (for simplicity, we assume $\mathbb H=H^{1/2}(\T,\R^{2n})$ endowed with the standard $H^{1/2}$-scalar product). We now recall the definition of the Conley index in this setting, referring to \cite{gip} for the details. On $\mathbb H$ we consider the canonical orthogonal splitting 
$$\mathbb H\equiv \mathbb H^+ \oplus \mathbb H^- \oplus \mathbb H^0$$
into positive/negative Fourier modes and constant loops. More precisely,
$$\mathbb H^\pm := \left \{ x = \sum_{k\in \N} e^{\pm 2\pi k Jt} x_k \ \Big | \ x_k \in \R^{2n}\right \}, \quad \mathbb H^0 \equiv \R^{2n},$$
where $J$ is the standard complex structure in $\R^{2n}$. We also set for $k,l\in \mathbb N$
$$\mathbb H_k := \Big \{e^{2\pi k Jt} q \ \Big |\ q \in \R^{2n} \Big \}, \quad \mathbb H^{-k,l} := \bigoplus_{i=-k}^l \mathbb H_i,$$
and, for any subset $A\subset \mathbb H$,
$$A^{-k,l}:= A\cap \mathbb H^{-k,l}.$$
We define 
$$L:\mathbb H\to \mathbb H, \quad Lx = L(x^++x^-+x^0) := x^+-x^-,$$
where $x^\pm,x^0$ are the orthogonal projections of $x\in \mathbb H$ onto $\mathbb H^\pm,\mathbb H^0$ respectively. A \textit{LS-vector field} $F$ on $\mathbb H$ is a gradient vector field of the form 
$$F=L+K$$
where $K$ is Lipschitz continuous and maps bounded sets into relatively compact sets. The flow $\varphi$ generated by a LS-vector field is called a \textit{LS-flow}. For any $k,l\in \N$ we finally set 
$$F^{-k,l}:= \big ( L + \pi^{-k.l}\circ K\big )\Big |_{\mathbb H^{-k,l}},$$
where $\pi^{-k,l}$ denotes the orthogonal projection onto $\mathbb H^{-k,l}$, and denote by $\varphi^{-k,l}$ the induced flow. 

Let $N\subset \mathbb H$ now be an isolating neighborhood for the LS-flow $\varphi$. As shown in \cite{gip}, $N$ yields an isolating neighborhood $N^{-k,l}$ for the flow $\varphi^{-k,l}$, provided $k,l$ are 
chosen so large that the scalar product of  $F= L+K$ and $L$ is positive on the orthogonal complement of $\mathbb H^{-k,l}$. Therefore, associated with the isolating neighborhood $N$ we have the family of Conley indices 
$$h^{-k,l}(N,\varphi) := h^{-k,l}(N^{-k,l},\varphi^{-k,l}), \quad \forall k,l \ \text{large enough}.$$ 
Such a family of Conley indices stabilizes in the following sense: if $k'>k$ and $l'>l$, then 
$$h^{-k',l'}(N,\varphi ) = S^{2n(l'-l)}\cdot h^{-k,l}(N,\varphi),$$
meaning that the Conley index $h^{-k',l'}(N,\varphi )$ is obtained from $h^{-k,l}(N,\varphi)$ by a suspension of dimension equal to the dimension of $\mathbb H^{l+1,l'}$. 


\subsection{Relative cup-length of the Conley index.} We introduce the notion of relative cup-length  in (locally) compact 
metric spaces following \cite{KGUss} (see also \cite{AIS}). 
Thus, let $A\subset X\subset Y$ be compact metric spaces. Denoting the Alexander-Spanier cohomology by $H^*$, it is easy to see that $H^*(X,A)$ has a natural structure of $H^*(Y)$-module with multiplication induced by the cup-product
$$\beta \cdot \alpha := \iota^* \beta \cup \alpha, \quad \forall \alpha \in H^*(X,A), \ \forall \beta \in H^*(Y),$$
where $\iota:X\to Y$ is the standard inclusion. We set the \textit{relative cup-length} $\text{RCL}(X,A;Y)$ of the $H^*(Y)$-module $H^*(X,A)$ to be equal to:
\begin{itemize}
\item $0$, if $H^*(X,A)=0.$
\item $1$, if $H^*(X,A)\neq 0$ and $\beta \cdot \alpha =0$ for all $\alpha \in H^*(X,A)$ and all $\beta \in H^{>0}(Y)$. 
\item $k\geq 2$, if there exists $\alpha_0\in H^*(X,A)$ and $\beta_1,...,\beta_{k-1}\in H^{>0}(Y)$ such that 
$$(\beta_1\cup ... \cup \beta_{k-1})\cdot \alpha_0  \neq 0,$$ 
and 
$$(\gamma_1\cup ... \cup \gamma_k)\cdot \alpha =0, \quad \forall \alpha \in H^*(X,A), \ \forall \gamma_1,...,\gamma_k \in H^{>0}(Y).$$
\end{itemize}
The next lemma follows directly by definition and naturality of the cup-product. 

\begin{lm}
Let $A\subset X\subset Y \subset Z$ be compact metric spaces. Then 
$$\mathrm{RCL}(X,A;Z)\leq \mathrm{RCL}(X,A;Y)$$
and
$$\mathrm{CL}(Y) \geq \mathrm{RCL}(X,A;Y) -1.$$
\label{lm:2}
\end{lm}

Fix now $\varphi$ a flow on a locally compact metric space. Suppose $\Omega$ is an isolating neighborhood for $\varphi$, and set $S:= \text{inv}(\Omega)$. We define the relative cup-length of $\Omega$ by 
$$\text{RCL}(\Omega,\varphi) := \text{RCL}(N,L;\Omega),$$ 
where $(N,L)$ is any index pair for the isolated invariant set $S$. Lemma 3.3 in \cite{KGUss} shows that the definition above is well-posed, that is, independent of the choice of the index pair $(N,L)$, and that the relative cup-length is invariant under continuation. For these reasons, we call $\text{RCL}(\Omega,\varphi)$ the \textit{relative cup-length of the Conley index}. 

For our purposes, the importance of the relative cup-length of the Conley index relies on the fact that it provides a lower bound on the number of critical points of a smooth function (for a more general statement involving Morse decompositions we refer to \cite{AIS}). 

\begin{thm}
Let $\varphi$ be the gradient flow (or, more generally, gradient-like) of some smooth function $f$, and let $\Omega$ be an isolating neighborhood for $\varphi$. Then, $f$ has at least $\mathrm{RCL}(\Omega,\varphi)$ critical points. 
\label{thm:2}
\end{thm}

\begin{proof}
A proof can be found in \cite{KGUss}. Another approach, which uses Morse-Floer theoretical methods, can be found in \cite{RSW}. For the reader's convenience we give here an alternative argument which is closer to 
Conley index theory. 

Thus, take $\epsilon>0$ such that the $\epsilon$-neighborhood $S_\epsilon$ of $S:= \text{inv}(\Omega)$ is contained in $\Omega$. Choosing an index pair in $S_\epsilon$, by the first inequality in Lemma \ref{lm:2} we infer that 
$$\text{RCL}(S_\epsilon,\varphi)\geq \text{RCL}(\Omega, \varphi).$$
Combining this with the second inequality in Lemma \ref{lm:2} we obtain that 
$$LS(S_\epsilon)\geq \text{CL}(S_\epsilon)+1 \geq \text{RCL}(S_\epsilon,\varphi),$$
where $LS(\cdot)$ denotes the Lyusternik-Schnirelmann category. By the very definition of the Lyusternik-Schnirelmann category, the invariant set $S$ can be covered by $LS(S)$ null-homotopic open sets, and so can be any small enough neighborhood. Therefore, $LS(S)\geq LS(S_\epsilon)$ for $\epsilon>0$ small enough. As $\varphi$ is the gradient flow of $f$ (more generally, gradient like for $f$), one easily sees that each of the null-homotopic sets above must contain at least one critical point of $f$. Therefore,
$$|\text{crit}(f)| \geq LS(S) \geq \text{RCL}(\Omega, \varphi),$$
thus completing the proof.
 \end{proof}
 
 Theorem \ref{thm:2} can be as well used to bound from below the number of rest points of an LS-flow $\varphi:\R\times \mathbb H\to \mathbb H$. Indeed, as shown in \cite{IzydorekJDE} (see also \cite{AIS}), we have that 
 $$|\text{crit}(f)| \geq \text{RCL}(\Omega^{-k,l}, \varphi^{-k,l})$$
provided  $k,l\in \mathbb N$ are  sufficiently large in the sense explained in the preceding subsection. 

\begin{rem}
The main reason behind the use of the relative cup-length instead of the cup-length is that the latter one is of no use in a Hilbert setting. In fact, by the very definition of the Conley index, the cup-length of the Conley index is always 0 since suspension by spheres kills the cup-length, while the relative cup-length remains the same. 
\end{rem}


\section{Proof of Theorem \ref{thm:a}}

The first step in the proof of Theorem \ref{thm:a} is to show that Condition (A) is satisfied for $\lambda_0=-1/2$. Equivalently, by the invariance properties of the second equation in \eqref{eq:hameq4}, we can show that 
$$\|\jmath^* \text{pr}_2 F_\lambda (x,z)\|\geq c \|y\|, \quad \forall (x,y)\in \mathbb H,$$
for every $\lambda \in [-1/2,-1/2+\epsilon) \cup (1/2-\epsilon,1/2]$, for suitable choice of $c,\epsilon>0$.

\begin{lm}
Assume that $H_0:\T\times \T^{2m}\times \mathbb C\mathbb P^n\to \R$ is such that $c:= \max\{2 \|H_0\|_\infty,\|\nabla_{\mathbb C\mathbb P^n} H_0\|_\infty\}<1/2$, and denote with $H$ the lift of $H_0$ to $\T\times \R^{2m}\times \R^{2n+2}$ as in Chapter 2. Then, under the notation of Chapter 2, there exist $\epsilon,\delta>0$ such that 
$$\|\jmath^* \mathrm{pr}_2 F_\lambda(x,y) \| \geq \delta \|y\|, \quad \forall (x,y)\in \mathbb H,$$
for every $\lambda \in [-\frac 12,-\frac 12 + \epsilon)\cup (\frac 12 - \epsilon,\frac 12].$ In other words, Condition $\mathrm{(A)}$ is satisfied with $\lambda_0=1/2$. 
\label{lem:conditiona}
\end{lm}

\begin{proof}
We consider the case $\lambda \leq \frac 12$ (being the case $\lambda \geq - \frac 12$ similar) and set 
$$w_\lambda := \jmath^* \big ( -J\dot y + 2\pi \lambda y) = L y  + 2\pi \lambda \jmath^*y, $$
so that $ \jmath^* \mathrm{pr}_2 F_\lambda(x,y) =\omega_\lambda - \jmath^* \nabla_y H(x,y)$. Observe that the Fourier coefficients of $\omega_\lambda$ are
$$\big |\langle w_\lambda,e_k\rangle \big | = \big | \text{sgn} (k) y_k + 2\pi \lambda \frac 1{2\pi |k|} y_k \big | \geq (1-\lambda)|y_k|, \quad \forall k\neq0,$$ 
and for $k=0$ by $\langle w_\lambda , e_0\rangle = 2\pi \lambda y_0$.
Since $\lambda < 1-\lambda$, we further compute
\begin{align*}
\|w_\lambda\|^2 = (2\pi \lambda)^2 |y_0|^2 + \sum_{k\neq 0 } 2\pi |k| \big |\langle \omega_\lambda,e_k\rangle \big |^2 &=  (2\pi \lambda)^2  |y_0|^2 + \sum_{k\neq 0 } 2\pi |k| (1-\lambda)^2 |y_k|^2 \\ 
&\geq \lambda^2 \big (|y_0|^2 +  \sum_{k\neq 0}2\pi  |k| |y_k|^2 \big ) = \lambda^2 \|y\|^2,
\end{align*}
On the other hand, by assumption we have that $\max_{|y|=1} |\nabla_y H(t,x,y)|=c$ and hence 
\begin{align*}
\|\jmath^* \nabla_y H(x,y)\|^2 &\leq \|\jmath^* \nabla_y H(x,y)\|^2_{H^1} = \|\nabla_yH(x,y)\|_{L^2}^2\\
&= \int_0^1 |\nabla_y H(t,x(t),y(t))|^2\, \mathrm d t  \leq c^2 \|y\|_{L^2}^2 = c^2 \|y\|^2.
\end{align*}
Putting all estimates together yields 
\begin{align*}
\|\jmath^* \mathrm{pr}_2 F_\lambda(x,y) \| & = \|\omega_\lambda - \jmath^* \nabla_y H(x,y)\| \geq \|\omega_\lambda\| - \| \jmath^* \nabla_y H(x,y)\| \geq \big (\lambda - c\big ) \|y\|
\end{align*}
and the claim follows taking $\epsilon >0$ so small that $\delta:=\frac 12 - \epsilon - c>0$.
\end{proof}

\begin{rem}
With more accurate estimates we could slightly weaken the condition $c<1/2$ in the lemma above. For instance, one can show that the same conclusion holds if $c<\sqrt{\pi}/2$ by considering the decomposition of $y$ into $y^0+y^\perp$, where $y^0$ is the constant part in the Fourier decomposition of $y$ and $y^\perp=y-y^0$. However, as Example \eqref{eq:counterexample} in the previous section shows, 
we cannot expect the statement to be true when $c$ is larger than $\pi/2$. 
\end{rem}

As in Section 2, we denote the projection of $\A_H$ to $\T^{2m}\times \R^{2n+2}\times \mathbb H^+\times \mathbb H^- \times \R$ again with $\A_H$. 
The next step is to show that the gradient flows of $\A_H$ and $\A_0$, namely the modified Hamiltonian action associated with the zero Hamiltonian, are related by continuation. In other words, there exists a bounded open set $\Omega\subset \T^{2m}\times \R^{2n+2}\times \mathbb H^+\times \mathbb H^- \times [-\frac 12 ,\frac 12]$ which is an isolating neighborhood for the gradient flow of $\mathbb A_{sH}$, for any $s\in [0,1]$. For notational convenience we write
 $$\mathbb H = \mathbb H_x \times \mathbb H_y = \T^{2m} \times  \mathbb H_x^\perp \times \mathbb H_y,$$
where $\mathbb H_x^\perp=H^{1/2}_0(\T,\R^{2m})$ is the Hilbert space of loops in $\R^{2m}$ having zero mean and $\mathbb H_y = H^{1/2}(\T,\R^{2n+2})$, and denote with $B_R(\cdot)$ the open ball with radius $R$.   

\begin{prop}
Under the assumptions of Lemma \ref{lem:conditiona} there exists $R>0$ such that 
$$\Omega_R := \T^{2m}\times B_R(\mathbb H_x^\perp)\times B_R (\mathbb H_y) \times [-\frac 12,\frac 12]$$
is an isolating neighborhood for the gradient flow of $\A_{sH}$, for any $s\in [0,1]$. In particular, the gradient flows of $\A_H$ and $\A_0$ are related by continuation within $\Omega_R$. 
\end{prop}
\begin{proof}
We first prove that, if $(x_k,y_k,\lambda_k)\subset \mathbb H\times [-\frac 12,\frac 12]$ is such that $\nabla \A_H(x_k,y_k,\lambda_k)\to 0$, then $(x_k,y_k,\lambda_k)$ is bounded. Without loss of generality 
we can assume that $\lambda_k\to \lambda_\infty$. If $\lambda_\infty\in [-\frac 12,-\frac 12 + \epsilon)\cup (\frac 12 - \epsilon,\frac 12]$, then by Lemma \ref{lem:conditiona} we infer that $y_k\to 0$ strongly in $H^{1/2}$. 
For the third component of $\nabla \A_H(x_n,y_n,\lambda_n)$ we thus obtain
$$\pi \big (\chi'(\lambda_k)\|y_k\|_2^2 -1 \big ) \to - \pi, \quad \text{for}\ k\to +\infty,$$
a contradiction. Therefore, we can assume $\lambda_\infty \in (-\frac 12 + \epsilon, \frac 12 - \epsilon)$. Using again the third component of $\nabla \A_H(x_k,y_k,\lambda_k)$ and Property (ii) in the definition of the function $\chi$ we infer that 
$$\pi \big (\chi'(\lambda_k)\|y_k\|_2^2-1\big ) = \pi \big (\|y_k\|_2^2-1\big ) \to 0,$$
that is, $\|y_k\|_{L^2}\to 1$. By construction of $H$ we further have 
$$\|\jmath^*\nabla_y H(x,y) \| \leq \|\nabla_y H(x,y)\|_{L^2} \leq c \|y\|_{L^2}, \quad \forall y,$$
and hence 
\begin{align*}
\|Ly_k\| =\| y_k^+-y_k^-\| & = \| \nabla_y \A_H(x_k,y_k,\lambda_k) - \jmath^* \nabla_y H(x_k,y_k) - 2\pi \lambda_k \jmath^* y_k\|\\
			&\leq \|\nabla_y \A_H(x_k,y_k,\lambda_k)\| + \| \jmath^* \nabla_y H(x_k,y_k) \| + 2\pi \lambda_k \|\jmath^*y_k\|\\
			&= O(1).
			\end{align*}
In particular, both $(y^+_k)$ and $(y_k^-)$ are bounded in $H^{1/2}$. Since $(y_k)$ is also bounded in $L^2$, we deduce that $(y_k)$ is bounded in $H^{1/2}$. By the continuous embedding of $H^{1/2}$ into all $L^p$-spaces 
(in particular, into $L^4$), the boundedness of $(y_k)$ in $H^{1/2}$ implies that the sequence $(\nabla_x H(x_k,y_k))$ is bounded in $L^2$, which in turn yields the boundedness of $(\jmath^* \nabla_x H(x_k,y_k))$ in $H^{1/2}$. 
From 
$$o(1) = \nabla_x \A_H(x_k,y_k,\lambda_k) = Lx_k - \jmath^*  \nabla_x H(x_k,y_k))$$
we finally deduce that $(x_k^+)$ and $(x_k^-)$ are bounded in $H^{1/2}$, and hence that $(x_k)$ is bounded in $H^{1/2}$ since $(x_k^0)\subset \T^{2m}$.

An easy inspection of the proofs of Lemma 3 and of the claim above show that both statements hold for the Hamiltonian $sH$, for any $s\in [0,1]$, and that the bounds on the $H^{1/2}$-norm of $x_n$ resp. $y_n$ are 
uniform in $s\in [0,1]$. The claim of the proposition now follows from standard arguments. For the reader's convenience we include a proof here. From what we have proved above we deduce that there exist $r,\rho>0$ such that 
$$\| \nabla \A_{sH} (x,y,\lambda) \| \geq \rho$$
for all $(x,y)\in \mathbb H$ such that $\|x^\perp\|\geq r$ or $\|y\|\geq r$, and for all $s\in [0,1]$. In particular, all critical points of $\A_{sH}$ are contained in 
$$ \Omega_r:= \T^{2m}\times B_r(\mathbb H_x^\perp)\times B_r (\mathbb H_y) \times [-\frac 12,\frac 12],$$
for some $r>0$. Let now $z(t)=(x_0(t),x^\perp(t),y(t),\lambda(t))$ be a gradient flow line for $\A_{sH}$ such that $z(0)=z_0 \in \Omega_r$ and $\A_{sH}(z(\cdot))\subset [a,b]$ for suitable $a,b\in \R$ (recall that $\A_{sH}$ is bounded on $\Omega_r$, uniformly in $s\in [0,1]$). Let $t_0\in \R$ be such that $\max \{\|x^\perp(t_0)\|, \|y(t_0)\|\}=r$, and assume that $z|_{[t_0,t_1]}$ is entirely contained in 
$$\T^{2m}\times \big (B_r(\mathbb H_x^\perp)\times B_r (\mathbb H_y) \big )^c\times [-\frac 12,\frac 12].$$
Then, we compute 
\begin{align*}
b-a &\geq \A_{sH}(z(t_1)) - \A_{sH}(z(t_0))\\
	&= \int_{t_0}^{t_1} \frac{\mathrm d}{\mathrm d t} \Big (\A_{sH}(z(t))\Big )\, \mathrm d t\\
	&= \int_{t_0}^{t_1} \mathrm d \A_{sH} ((t)) [\dot z(t)]\, \mathrm d t\\
	&=\int_{t_0}^{t_1} \|\nabla \A_{sH}(z(t))\|^2\, \mathrm d t\\
	&\geq (t_1-t_0)\rho^2 
\end{align*}
which implies that $t_1-t_0\leq \frac{b-a}{\rho^2}$. In particular, for the norm of $x^\perp(t_1)$ (a similar estimate holding as well for the norm of $y(t_1)$) 
\begin{align*}
\|x^\perp(t_1)\| &\leq \|x(t_0)\| + \|x(t_1)-x(t_0)\| \\
		& \leq r + \|z(t_1)-z(t_0)\| \\
		&\leq r + \int_{t_0}^{t_1} \|\dot z(t)\|\, \mathrm d t\\
		&\leq r + (t_1-t_0)^{1/2} \Big (\int_{t_0}^{t_1} \|\dot z(t)\|^2\, \mathrm dt \Big )^{1/2}\\
		&= r + (t_1-t_0)^{1/2} \Big (\int_{t_0}^{t_1} \|\nabla \A_{sH}(z(t))\|^2\, \mathrm d t \Big )^{1/2}\\
		& \leq r + (t_1-t_0)^{1/2} (b-a)^{1/2}\\
		&\leq r + \frac{b-a}{\rho}.
\end{align*}
The claim follows then taking for instance $R:= 2 \big (r + \frac{b-a}{\rho}\big )$.
\end{proof}

\begin{proof}[Proof of Theorem \ref{thm:a}] 
Proposition 1 shows that, under the assumptions of Theorem \ref{thm:a}, the gradient flows of $\A_H$ and $\A_0$ are related by continuation within $\Omega_R$. By continuation invariance, it therefore suffices to compute the relative cup-length of the Conley index for the zero-Hamiltonian. Since the gradient vector field $\nabla \A_0$ is of product form, the relative cup-length of the Conley index is given by one minus the sum of the relative cup-length of the two factors (the ``$\T^{2m}$'' and the ``$\mathbb C\mathbb P^n$'' direction). The first one gives the $2m+1$ contribution (see \cite{SW}) whereas the latter gives the $n+1$ contribution (for a proof see e.g. \cite{AIS}), thus finishing the proof. 
\end{proof}

\end{document}